\documentclass{article}
\usepackage{spconf,amsmath,graphicx}

\usepackage{enumerate}
\usepackage{amsmath, amsthm, amssymb}
\usepackage{float}
\usepackage{subcaption}
\usepackage{flexisym}
\usepackage{mathtools}
\usepackage{breqn}
\usepackage{array}
\usepackage{tikz}
\usepackage{cite}
\usetikzlibrary{shapes.geometric, arrows}
\usepackage[font = small]{caption}
\usepackage{paralist}
\usepackage{mathrsfs}
\usepackage{algorithm}
\usepackage{algpseudocode}
\usepackage[inline]{enumitem}
\usepackage{bbm}
\allowdisplaybreaks[1]

\newtheorem{lemma}{Lemma}
\newtheorem{theorem}{Theorem}

\theoremstyle{definition}

\DeclareMathOperator*{\diag}{diag}
\DeclareMathOperator*{\blkdiag}{blkdiag}

\title{Resilient Distributed Recovery of Large Fields}
\name{Yuan Chen, Soummya Kar, and Jos\'{e} M. F. Moura \thanks{\noindent This material is based upon work supported by the Department of Energy under award number DE-OE0000779, by DARPA under agreement numbers FA8750-12-2-0291 and HR00111320007, and by the National Science Foundation under award numbers CCF1513936 and CNS1837607.}}
\address{Carnegie Mellon University \\
	Department of Electrical and Computer Engineering \\
	Pittsburgh, PA 15213 USA}
\begin{document}
\ninept
\maketitle
\begin{abstract}
This paper studies the resilient distributed recovery of large fields under measurement attacks, by a team of agents, where each measures a small subset of the components of a large spatially distributed field. An adversary corrupts some of the measurements.
The agents collaborate to process their measurements, and each is interested in recovering only a fraction of the field.  We present a field recovery \textit{consensus+innovations} type distributed algorithm that is resilient to measurement attacks, where an agent maintains and updates a local state based on its neighbors states and its own measurement. Under sufficient conditions on the attacker and the connectivity of the communication network, each agent's state, even those with compromised measurements, converges to the true value of the field components that  it is interested in recovering. Finally, we illustrate the performance of our algorithm through numerical examples.
\end{abstract}

\section{Introduction}\label{sect: intro}
In many applications in the Internet of Things (IoT), device instrument a large environment and measure a spatially distributed field. For example, a network of roadside units measures traffic patterns throughout a city~\cite{TrafficRSU}, and teams of mobile robots collaborate to map and navigate unknown environments~\cite{RobotMapping}. The devices need to process their   measurements to extract useful information about the physical field.  IoT devices, however, are vulnerable to cyber attack~\cite{ChenIoT, ZhangIoT}. Without proper security countermeasures, adversaries may hijack individual devices, manipulate their measurements, and prevent them from achieving their computation objectives. 

This paper studies the distributed recovery of large physical fields 
under measurement attacks. The agents or devices make measurements of the unknown field in their proxomity, and process their measurements to recover the value of the field. Due to the field's large size, no individual agent seeks to recover the \textit{entire} field. Instead, each agent seeks to recover a subset of the field components. For example, in multi-robot navigation, an individual robot attempts to recover just its local surroundings instead of mapping the entire environment. The devices are unable to recover their desired components of the field using just their local measurements; they share information over a communication network to accomplish their processing objectives, but an adversary may attempt to thwart this goal by arbitrarily manipulating a subset of the measurements. Each agent's goal is to process its measurement and information from its network neighbors to recover specific components of the field without being misled by the adversary. 

We present a \textit{consensus+innovations} type algorithm~\cite{Kar1, Kar2} for resilient field recovery. Device maintain and update a local state (an estimate of the field components of interest) based on the state of its neighbors in the network and its own measurements. When updating its state, each device applies an adaptive state dependent gain to its own measurements to mitigate the effects of potential measurement attacks. We show that, under sufficient conditions on the compromised measurements and on the connectivity of the communication network, our algorithm ensures that all of the agents' local states converge to the true values of their desired field components.


Prior work in resilient computation has focused on settings where all devices or agents share a common processing objective. For example, in resilient consensus, agents attempt to reach agreement on a decision or value in the presence of adversaries~\cite{ByzantineGenerals, ResilientConsensus, LeBlanc1} and in resilient parameter estimation, agents attempt to recover a common unknown parameter from local measurements while coping with malicious data~\cite{LeBlanc2, ChenDistributed2, ChenSAGE}. In contrast, in resilient field recovery, agents have different, heterogeneous processing objectives. This makes the problem more challenging: when communicating with neighbors, agents must further process their neighbors' messages to extract information relevant to its own objectives.

Existing work has studied field recovery in \textit{nonadversarial} environments.  In~\cite{Cortes}, the authors design a procedure to optimally place sensors in a spatially correlated field. Reference~\cite{SahuRandomFields} studies distributed recovery of static fields, and reference~\cite{KhanKF} designs distributed Kalman Filters for estimating very large time-varying random fields. None of these references, however, address field recovery in \textit{adversarial} scenarios. In contrast, this paper presents an algorithm for \textit{resilient} field recovery under measurement attacks. The rest of this paper is organized as follows.  Section~\ref{sect: background} reviews the measurement and attack models and formalizes the field recovery problem. We present our resilient distributed field recovery algorithm in Section~\ref{sect: algorithm} and analyze its performance in Section~\ref{sect: analysis}. Section~\ref{sect: examples} illustrates the performance of our algorithm through numerical examples, and we conclude in Section~\ref{sect: conclusion}.

\textit{Notation}: Let $\mathbb{R}^k$ be the Euclidean space of dimension $k$, $I_k$ the $k$ by $k$ identity matrix, and $\mathbf{1}_k$ the column vector of $k$ ones. The $j^{\text{th}}$ canonical basis vector of $\mathbb{R}_k$ ($j = 1, \dots, k$) is $e_j$, a column vector with $1$ in the $j^{\text{th}}$ element and $0$ elsewhere. For symmetric matrices $A = A^\intercal$, $A \succ 0$ ($A \succeq 0$) means that $A$ is positive definite (semidefinite). For a matrix $A$, $[A]_{i, j}$ is the element in the $i^{\text{th}}$ row and $j^{\text{th}}$ colunm. For a vector $v$, $[v]_{i}$ is the $i^{\text{th}}$ element. A simple undirected graph $G = (V, E)$ has vertex set $V = \left\{ 1, \dots, N \right\}$ and edge set $E$. Each vertex $n$ has neighborhood $\Omega_n$ (the set of vertices that share an edge with vertex $n$) and degree $d_n = \left\lvert \Omega_n \right\rvert$. The degree matrix of $G$ is $D = \diag\left( d_1, \dots, d_N \right)$, the adjacency matrix is $A$, where $[A]_{n, l} = 1$ if there is an edge between vertex $n$ and vertex $l$ and $[A]_{n, l} = 0$ otherwise, and the Laplacian matrix is $L = D - A$. The Laplacian matrix has ordered eigenvalues $\lambda_1(L) \leq \lambda_2 (L) \leq \cdots \leq \lambda_N(L)$, where $\lambda_1(L) = 0$. For connected graphs $\lambda_2(L) > 0$. References~\cite{Spectral, ModernGraph} review spectral graph theory.

\section{Background}\label{sect: background}
Consider $N$ agents or devices measuring an unknown field collected in the parameter $\theta^* \in \mathbb{R}^M$. The field parameter $\theta^*$ is high dimensional and spatially distributed over a large physical area; for example, in the context of multi-robot navigation, it may represent the location of obstacles in a large unknown environment. 
In normal operation conditions (i.e., in the absence of measurement attacks), each agent's measurement is
\begin{equation}\label{eqn: measurementModel}
	y_n = H_n \theta^*.
\end{equation}
The measurements may have different dimensions across agents: we let $P_n \ll M$ be the dimension of agent $n$'s measurement (i.e., $y_n(t) \in \mathbb{R}^{P_n}$ and $H_n \in \mathbb{R}^{P_n \times M}$). 
An adversary changes arbitrarily manipulates a subset of the measurement values. We model the effect of the attack with the additive disturbance $a_n$:
\begin{equation}\label{eqn: attackedMeasurement}
	y_n = H_n \theta^* + a_n.
\end{equation}
In this paper, we focus on distributed field recovery from a single snapshot of (noiseless) measurements at each agent. The case of noisy measurement streams (i.e., \textit{sequences} of measurements over time $y_n(0), y_n(1), \dots$) is the subject of our ongoing work~\cite{ChenSAFE}. 

We use the convention from~\cite{ChenSAGE} for indexing scalar measurement globally across all agents. Let
\begin{equation}\label{eqn: stackedMeasurement}
	\mathbf{y}_t = \left[ \begin{array}{ccc} y_1^\intercal & \cdots & y_N^\intercal \end{array} \right]^\intercal = \mathcal{H} \theta^* + \mathbf{a},
\end{equation}
be the vector of all (scalar) measurements at time $t$, where $\mathcal{H} = \left[\begin{array}{ccc} H_1^\intercal & \cdots & H_N^\intercal \end{array} \right]^\intercal$ stacks the measurement matrices $H_1, \dots, H_N$, and, similarly, $\mathbf{a}$ stacks the measurement attacks. The stacked measurement $\mathbf{y}$ has dimension $P = \sum_{i = 1}^N P_n$. We label the individual components of $\mathbf{y}$ and the rows of $\mathcal{H}$ from $1$ to $P$:
	$\mathbf{y} = \left[ \begin{array}{ccc} y^{(1)} & \cdots & y^{(P)} \end{array} \right]^\intercal,$
	$\mathcal{H} = \left[\begin{array}{ccc} h_1 & \cdots & h_p \end{array} \right]^\intercal.$ 
From the above indexing convention, we assign the indices $\overline{P}_n + 1, \dots, \overline{P}_n + P_n$, where $\overline{P}_n = \sum_{j = 0}^{n-1} P_j$, to the individual components of $y_n$, the individual components of $a_n$ and rows of $H_n$ (of agent $n$). We assume that every row of $\mathcal{H}$ is nonzero and has unit $\ell_2$-norm, i.e., $\left\lVert h_p \right\rVert_2 =1 $ for all $p = 1. \dots, P$. The set $\mathcal{A} = \left\{ p \in \left\{1, \dots, P \right\} \vert  a^{(p)} \neq 0 \right\}$ is the set of compromised measurements, and $\mathcal{N} = \left\{ 1, \dots, P \right\} \setminus \mathcal{A}$ is the set of uncompromised measurements. The agents do not know which measurements are compromised.

The field parameter $\theta^*$ is high dimensional and spatially distributed over a large physical area, so, each agent's measurement is only physically coupled to a few components of $\theta^*$. 
The measurement matrices $H_n$ capture the {physical coupling} between the field $\theta^*$ and the measurements $y_n(t)$. We define the \textit{physical coupling set}
\begin{equation}\label{eqn: couplingSet}
	\widetilde{\mathcal{I}}_n = \left\{ m \in \left\{1, \dots, M \right\} \vert H_n e_m \neq 0\right\},
\end{equation}
where $e_m$ is the $m^{\text{th}}$ canonical basis vector of $\mathbb{R}^M$, as the indices of the nonzero \textit{columns} of the matrix $H_n$. The set $\widetilde{\mathcal{I}}_n$ describes all components of $\theta^*$ that are physically coupled to the measurement at agent $n$.

In distributed field recovery, each agent is interested in recovering a subset of components of $\theta^*$. This contrasts with the setup of distributed \textit{parameter} estimation~\cite{ChenDistributed2, ChenSAGE}, where each agent is interested in estimating all components of $\theta^*$. In the context of robotic navigation, for example, each robot may only be interested in estimating its local surroundings instead of the entirety of the (large) unknown environment. 
For agent $n$, its \textit{interest set} $\mathcal{I}_n$ is the set of components of $\theta^*$ that it wishes to recover, sorted in ascending order. Following the convention from~\cite{SahuRandomFields}, the expression $\mathcal{I}_n \left(r\right) = m$ means that the $r^{\text{th}}$ element of the $\mathcal{I}_n$ (for $r = 1, \dots, \left\lvert \mathcal{I}_n \right\rvert$) is the component $m$ (i.e., the $m^{\text{th}}$ component of $\theta^*$). Conversely, $\mathcal{I}_n^{-1} \left( m \right) = r$ means that component $m$ is the $r^{\text{th}}$ element in $\mathcal{I}_n$. 

We require each agent's physical coupling set to be a subset of its interest set, i.e.,
	$\widetilde{\mathcal{I}}_n \subseteq \mathcal{I}_n,$
for all $n = 1, \dots, N$. 
In addition to each agent's interest set, we also define, for each \textit{component} of $\theta^*$ $m = 1, \dots, M$,
	$\mathcal{J}_m = \left\{n \in \left\{1, \dots, N \right\} \vert m \in \mathcal{I}_n \right\},$
the set of all agents interested in recovering component $m$. We assume that, for all $m = 1, \dots, M$, the set $\mathcal{J}_m$ is nonempty. 

Each individual agent may not have enough information from its local measurements alone to recover all components in its interest set. 
The agents exchange information over a communication network, modeled as an undirected graph $G = (V, E)$, where the vertex set $V$ is the set of agents and the edge set $E$ represents communication links between agents. For each component $m = 1, \dots, M$ of the field $\theta^*$, let $G_m$ be the graph induced by all agents interested in recovering $m$, i.e., all agents in $\mathcal{J}_m$. We assume that each subnetwork $G_m$ is connected for all $m$. 

An important concept in field and parameter recovery is global observability. We assume that the set of all measurements is globally observable for $\theta^*$: the observability Grammian matrix
	$\mathcal{G} = \mathcal{H}^\intercal\mathcal{H} = \sum_{p = 1}^P h_p h_p^\intercal$
is invertible. Global observability means that, using the stacked measurement vector $\mathbf{y}_t$, it is possible to exactly determine the value of $\theta^*$. Global observability is required for a fusion center, which collects and simultaneously processes the measurements of \textit{all} of the agents, to recover $\theta^*$, so we assume it here for a fully distributed setting.



\section{Resilient Distributed Field Recovery}\label{sect: algorithm}
In this section, we present a distributed field recovery algorithm that is resilient to measurement attacks. 

\subsection{Algorithm Description}\label{sect: algorithmDescription}
Each agent $n$ maintains a $\left\lvert {\mathcal{I}_n} \right\rvert$-dimensional state $x_n(t)$, where the $i^{\text{th}}$ component $\left[x_n(t) \right]_i$ is an estimate of $\left[\theta^*\right]_{\mathcal{I}_n(i)}$. Initially, each agent sets its state as $x_n(t) = 0$ and iteratively updates its state according to the following procedure.

\noindent \textbf{Step 1 -- Communication}: Each agent $n$ sends its current state $x_n(t)$ to each of its neighors $l \in \Omega_n$.

\noindent \textbf{Step 2 -- Message Censorship}: 
To account for different interest sets, each agent processes the states received from it neighbors. First, each agent $n$, for each of its neighbors $l \in \Omega_n$, constructs a \textit{censored} state $x^c_{l, n} (t)$ component-wise as follows:

\begin{equation}\label{eqn: censorReceived}
	\!\left[x^c_{l, n}(t) \right]_i = \left\{ \begin{array}{ll} \left[x_{l}(t)\right]_{\mathcal{I}_l^{-1} \left(\mathcal{I}_n(i)\right)}, \!\! \!\!& \text{if } \mathcal{I}_n(i) \in \mathcal{I}_l,\!\! \\ 0, &\text{otherwise,} \end{array} \right.
\end{equation}
for $i = 1, \dots, \left\lvert \mathcal{I}_n \right \rvert$. 
Second, for each of its neighbors $l \in \Omega_n$, agent $n$ also constructs a processed version of its own state, $x^p_{l, n}(t)$, component-wise as follows:

\begin{equation}\label{eqn: censorSelf}
	\left[x^p_{l, n}(t) \right]_i = \left\{ \begin{array}{ll} \left[x_{n}(t)\right]_i, & \text{if } \mathcal{I}_n(i) \in \mathcal{I}_l, \\ 0, &\text{otherwise.} \end{array} \right.
\end{equation}

\noindent \textbf{Step 3 -- State Update}: Each agent $n$ updates its state following
\begin{equation}\label{eqn: estimateUpdate}
\begin{split}
	x_n(t+1) = &x_n(t) - \beta_t \!\!\sum_{l \in \Omega_n(t)}\!\! {\left(x_{l, n}^p(t) - x_{l, n}^c(t)\right)}\\
	& + \alpha_t {H_n^c}^\intercal K_n(t) \left( y_n - H_n^c x_n(t)\right),
\end{split}
\end{equation}
where the matrix $H_n^c \in \mathbb{R}^{P_n \times \left\lvert \mathcal{I}_n \right\rvert}$ is the matrix $H_n$ after removing all \textit{columns} whose indices are not in $\mathcal{I}_n$, $K_n(t)$ is a diagonal gain matrix to be defined shortly, and $\alpha_t$ and $\beta_t$ are decaying weight sequences of the form 
	$\alpha_t = \frac{a}{(t+1)^{\tau_1}}, \: \beta_t = \frac{b}{(t+1)^{\tau_2}}.$
We select the scalar hyperparmeters $a, b, \tau_1, \tau_2$ to satisfy $a, b > 0$ and $0 < \tau_2 < \tau_1 < 1.$ The gain matrix $K_n(t)$ is defined as
\begin{equation}\label{eqn: KnDef}
	K_n(t) = \diag \left(k_{\overline{P}_n + 1}(t), \dots, k_{\overline{P}_n + P_n}(t) \right),
\end{equation}
where, for $p = \overline{P}_n + 1, \dots, \overline{P}_n + P_n$,
\begin{equation}\label{eqn: smallK}
	k_p (t) = \min\left(1, \gamma_t \left\lvert y^{(p)} - {h^c_p}^\intercal x_n(t) \right \rvert^{-1} \right),
\end{equation}
${h_p^c}^\intercal$ is the row vector $h_p^\intercal$ after removing all components not in $\mathcal{I}_n$, and $\gamma_t$ is a decaying threshold sequence of the form
	$\gamma_t = \frac{\Gamma}{(t+1)^{\tau_\gamma}}.$
We select the scalar hyperparameters $\Gamma$ and $\tau_\gamma$ to satisfy $\Gamma > 0$ and $0 < \tau_\gamma < \tau_1 - \tau_2$. 

The gain matrix $K_n(t)$ saturates the magnitude of each component of its innovation, $y_n - H_n^c x_n(t)$, at the threshold level, $\gamma_t$. The threshold $\gamma_t$ decays over time, which decreases the amount by which the innovation is able to influence the state update. Intuitively, this means that, initially, the agents are more willing to trust measurements that differ greatly from their current estimates of the field. As the agents update their states, they expect their estimates to move closer to the true values of the field, and they become less willing to trust measurements that differ greatly from their current estimates. By decaying the threshold over time, the agents prevent their states from being led astray by compromised measurements while still incorporating enough information from the uncompromised measurements to recover their desired field components.

Compared with algorithms for resilient distributed parameter recovery or estimation from~\cite{ChenSAGE, ChenDistributed2}, the algorithm in this paper introduces the additional message censorship step. This additional step is required because, unlike parameter recovery or estimation, where all of the agents are interested in recovering \textit{all} components of $\theta^*$, in distributed field recovery, each agent is only interested in recovering a subset of the components of $\theta^*$. Through message censorship, each agent $n$ extracts information from their neighbors' states about only the components of $\theta^*$ it is interested in recovering. 

\subsection{Main Result: Algorithm Performance}
For each agent $n$, let $\theta^*_{\mathcal{I}_n}$ be the $\left\lvert \mathcal{I}_n \right\rvert$-dimensional vector that collects all components of $\theta^*$ in which it is interested in recovering. We express $\theta^*_{\mathcal{I}_n}$ component-wise as
	$\left[ \theta^*_{\mathcal{I}_n} \right]_i = \left[\theta^* \right]_{\mathcal{I}_n(i)}$
for all $i = 1, \dots, \left\lvert \mathcal{I}_n \right\rvert$. The following theorem characterizes the behavior of the agents' local states under our resilient distributed field recovery algorithm. 
\begin{theorem}\label{thm: main}
Let $\mathcal{A}=\left\{p_1, \dots, p_{\left\lvert \mathcal{A} \right\rvert} \right\}$ be the set of compromised measurements, and let $\mathcal{H}_{\mathcal{A}} = \left[\begin{array}{ccc} h_{p_1} & \cdots & h_{p_{\left\lvert \mathcal{A} \right\rvert}} \end{array} \right]^\intercal$ be the matrix that collects all rows of $\mathcal{H}$ indexed by elements in $\mathcal{A}$. If the matrix 
	$\mathcal{G}_{\mathcal{N}} = \sum_{p \in \mathcal{N}} h_p h_p^\intercal$
satisfies
\begin{equation}\label{eqn: resilienceCondition}
	\lambda_\min \left(\mathcal{G}_{\mathcal{N}} \right) > \Delta_{\mathcal{A}},
\end{equation}
where $\mathcal{N} = \mathcal{P}\setminus \mathcal{A}$ and
\begin{equation}\label{eqn: maxDisturbance}
	 \Delta_{\mathcal{A}} =  \max \limits_{v \in \mathbb{R}^{\left\lvert \mathcal{A} \right\rvert}, \left\lVert v \right \rVert_\infty \leq 1} \left\lVert \mathcal{H}_{\mathcal{A}}^\intercal v \right\rVert_2,
\end{equation}
then, under the algorithm described in Section~\ref{sect: algorithmDescription}, we have
\begin{equation}\label{eqn: localConsistency}
	\lim_{t \rightarrow \infty} \left(t+1 \right)^{\tau_0} \left\lVert x_n(t) - \theta^*_{\mathcal{I}_n} \right\rVert_2 = 0,
\end{equation}
for all agents $n$ and every $0 \leq \tau_0 < \tau_\gamma.$
\end{theorem}

Theorem~\ref{thm: main} states that, as long as the resilience condition in~\eqref{eqn: resilienceCondition} is satisfied, the agents' local states converge to the true values of the field components they are interested in recovering. The sufficient resilience condition~\eqref{eqn: resilienceCondition} is a condition on the redundancy of uncompromised measurements and intuitively means that the uncompromised measurements (across all agents) should collectively measure $\theta^*$ redundantly enough to overcome the influence of the compromised measurements. The resilience of our field recovery algorithm depends on the redundancy of the measurements and not on the agents' interest sets. In addition to the resilience condition~\eqref{eqn: resilienceCondition}, we also have requirements on the topology of the communication network $G$. Recall, from Section~\ref{sect: background}, that, for each component $m = 1, \dots, M$, we require $G_m$, the subgraph of the communication network $G$ induced by the agents interested in recovering component $m$, to be connected. 

\section{Performance Analysis}\label{sect: analysis}
This section analyzes the performance of our resilient distributed field recovery algorithm. Due to space limitations, we outline the main steps, presented as intermediate results, to prove Theorem~\ref{thm: main}, and we provide a detailed analysis in the appendix. The performance anlaysis follows three main steps. First, we transform each agent's state into an auxiliary state. Second, we show that all of the agents' auxiliary states converges to a generalized network average state. Finally, we show that the generalized network average state converges to the true value of the field.

For each agent $n$, we define the auxiliary state $\widetilde{x}_n(t) \in \mathbb{R}^M$ (recall that $M$ is the dimension of the field $\theta^*$) component-wise as
\begin{equation}\label{eqn: tildeX}
	\left[ \widetilde{x}_n(t) \right]_i = \left\{ \begin{array}{ll} \left[ x_n(t) \right]_{\mathcal{I}^{-1}_n \left( i \right)}, & \text{if } i \in \mathcal{I}_n, \\ 0, &\text{otherwise,} \end{array} \right.
\end{equation}
for each $i =1, \dots, M$. Let $\mathbf{\widetilde{x}}_t = \left[ \begin{array}{ccc} \widetilde{x}_1(t)^\intercal & \cdots & \widetilde{x}_N(t)^\intercal \end{array}\right]^\intercal$ stack the auxiliary states of all the agents. We may show that $\widetilde{\mathbf{x}}_t$ evolves according to
\begin{equation}\label{eqn: stackedTildeX}
\begin{split}
	\widetilde{\mathbf{x}}_{t+1} &=\widetilde{\mathbf{x}}_t - \beta_t \mathbf{L} \widetilde{\mathbf{x}}_t + \alpha_t D_H^\intercal \mathbf{K}_t \left(\mathbf{y} - D_H \widetilde{\mathbf{x}}_t \right),
\end{split} 
\end{equation}
where $\mathbf{K}_t =  \blkdiag\left(K_1(t), \dots, K_N(t) \right),$ $D_H = \blkdiag (H_1, \allowbreak \dots, H_N )$, and the $NM \times NM$ matrix $\mathbf{L}$ is defined blockwise as follows. Let $\left[\mathbf{L}\right]_{n, l} \in \mathbb{R}^{M \times M}$ be the $(n, l)$-th \textit{sub-block} of $\mathbf{L}$, for $n, l = 1, \dots, N$, which is defined as
\begin{equation}\label{eqn: ltDef}
	\left[\mathbf{L}\right]_{n, l} = \left\{ \begin{array}{ll}-Q_n \sum_{i = 1:i \neq n}^N \left[L\right]_{n, i} Q_i, & \text{if } n = l, \\ \left[L \right]_{n, l} Q_n Q_l, & \text{otherwise,}\end{array} \right.
\end{equation}
where, for each agent $n$, the matrix $Q_n$ is an $M \times M$ diagonal matrix where the $m^{\text{th}}$ diagonal element (for $m = 1, \dots M$) is $1$ if agent $n$ is interested in recovering $\left[\theta^*\right]_m$ (i.e., $m \in \mathcal{I}_n$) and $0$ otherwise.
In~\eqref{eqn: ltDef}, the term $\left[ L \right]_{n, l}$ refers to the $(n, l)$-th (scalar) \textit{element} of the Laplacian $L$. We now use the auxiliary state update~\eqref{eqn: stackedTildeX} to analyze the performance of the distribute field recovery algorithm.

Let 
	$\overline{\mathbf{x}}_t = \mathcal{D} \left( \mathbf{1}_N^\intercal \otimes I_M \right)\widetilde{\mathbf{x}}_t$,
where $\mathcal{D} = \diag(\left\lvert{\mathcal{J}_1}\right\rvert^{-1}, \dots, \allowbreak \left\lvert{\mathcal{J}_M}\right\rvert^{-1})$, be the \textit{generalized network average state}. Each component $m = 1, \dots, M$ of $\overline{\mathbf{x}}_t$ is the average estimate of $\left[\theta^*\right]_m$ taken over all agents interested in recovering that specific component (i.e., all agents $n \in \mathcal{J}_m$). Define the matrix $\mathcal{Q} = \blkdiag\left(Q_1, \dots, Q_N \right)$. Then, we may show the following result:
\begin{lemma}\label{lem: resilientConsensus}
	Under the algorithm from section~\ref{sect: algorithmDescription},
\begin{equation}\label{eqn: resilientConsensus}
	 \lim_{t \rightarrow \infty} \left(t+1\right)^{\tau_3} \left\lVert \mathcal{Q} \left(\widetilde{\mathbf{x}}_t - \left(\mathbf{1}_N \otimes I_M \right)\overline{\mathbf{x}}_t \right) \right\rVert_2 = 0,
\end{equation}
for every $0 \leq \tau_3 < \tau_\gamma + \tau_1 - \tau_2$. 
\end{lemma}
\noindent Lemma~\ref{lem: resilientConsensus} states that, under the our algorithm, the auxiliary state of each agent $n = 1, \dots, N$, $\widetilde{x}_n(t)$ converges to the generalized network average state $\overline{\mathbf{x}}_t$ on all components the interest set $\mathcal{I}_n$.

We now study the behavior of the generalized network average state $\overline{\mathbf{x}}_t$. We are interested in 
	$\overline{\mathbf{e}}_t = \overline{\mathbf{x}}_t - \theta^*,$
the generalized network average state error. We may show the following result:
\begin{lemma}\label{lem: avg2}
	Under the algorithm from section~\ref{sect: algorithmDescription},  as long as $\lambda_{\min}\left(\mathcal{G}_{\mathcal{N}} \right) > \Delta_{\mathcal{A}},$ then
\begin{equation}\label{eqn: averageConv}
	\lim_{t \rightarrow \infty} (t+1)^{\tau_0} \left\lVert \overline{\mathbf{e}}_t \right \rVert_2 = 0,
\end{equation}
for every $0 \leq \tau_0 < \tau_\gamma$.  
\end{lemma}
\noindent Lemma~\ref{lem: avg2} states that, under the resilience condition  $\lambda_{\min}\left(\mathcal{G}_{\mathcal{N}} \right) > \Delta_{\mathcal{A}}$~\eqref{eqn: resilienceCondition}, the generalized network average state converges to the true value of the field $\theta^*$.

Using Lemmas~\ref{lem: resilientConsensus} and~\ref{lem: avg2}, we provide a proof sketch of Theorem~\ref{thm: main}. From the triangle inequality, we have \begin{equation}\label{eqn: mainTriangle}\left\lVert x_n(t) - \theta^*_{\mathcal{I}_n} \right\rVert_2 \leq \left\lVert \mathcal{Q} \widehat{\mathbf{x}}_t \right\rVert_2 + \left\lVert \overline{\mathbf{e}}_t \right\rVert_2.\end{equation} Then, from Lemmas~\ref{lem: resilientConsensus} and~\ref{lem: avg2}, under the resilience condition~\eqref{eqn: resilienceCondition}, we have
	 $\lim_{t \rightarrow \infty} \left(t+1\right)^{\tau_3} \left\lVert \mathcal{Q} \widehat{\mathbf{x}}_t \right\rVert_2 = 0$, and 
	 $\lim_{t \rightarrow \infty} (t+1)^{\tau_0} \left\lVert \overline{\mathbf{e}}_t \right \rVert_2 = 0 $ 
for every $0 \leq \tau_3 < \tau_\gamma + \tau_1 - \tau_2$ and every $0 \leq \tau_0 < \tau_\gamma$. Substituting the above relationships into~\eqref{eqn: mainTriangle} yields the desired result~\eqref{eqn: localConsistency}. 

\section{Numerical Examples}\label{sect: examples}
We consider a mesh network of $N = 400$ robots agents sensing an unknown environment, modeled by a two dimensional 230 unit by 230 unit grid. We assign each grid square a $[0, 255]$-valued state variable that represents the safety or occupancy of that particular location. For example, a state value of $0$ may represent a location that is completely free of obstacles, and a state value of $255$ may represent an impassable obstacle. The field parameter $\theta^*$ is the collection of the $52,900$ state values from each location. Figure~\ref{fig: network} visualizes the true field: the $x$ and $y$ axis represent location coordinates and the $z$ axis ($\theta^*$ axis) gives the state value of each $(x, y)$ coordinate location. 
\begin{figure}[h!]
	\centering
	\includegraphics[width = 0.7\columnwidth]{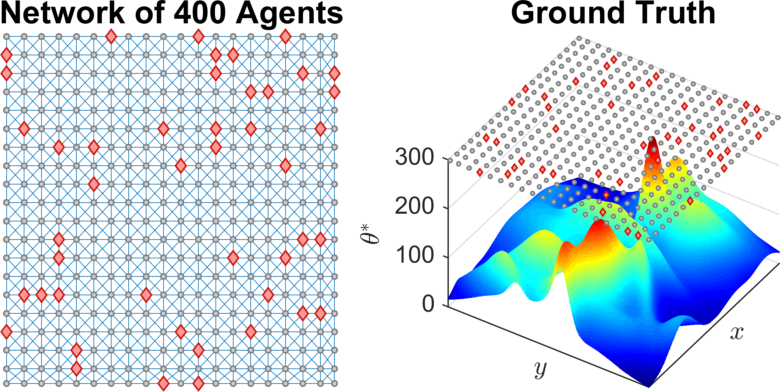}
	\caption{A mesh network of $400$ agents (left) placed in a $230 \times 230$ two dimensional grid environment (right). 
Agents represented by red diamonds have compromised measurements.}\label{fig: network}
\end{figure}

Each agent measures the state values of all locations in a $37 \times 37$ unit square subgrid centered at its location and is interested in recovering the state values of all locations in a $73 \times 73$ unit subgrid centered at its location. An adversary attacks $45$ of the agents and changes all (scalar) measurements of each agent under attack to $y^{(p)} = 255$. We compare the performance of our field recovery algorithm against the algorithm from~\cite{SahuRandomFields}, \textbf{CIRFE}, which does not account for measurement attacks, using the following hyperparameters: $a = 1, b = 0.084, \tau_1 = 0.26, \tau_2 = 0.001, \Gamma = 40, \tau_\gamma = 0.25.$

Figure~\ref{fig: performance1} shows that the measurement attack induces a persistent local recovery error under~\textbf{CIRFE}, while, under our resilient algorithm, each agent's recovery error converges to $0$ despite the attack.
\begin{figure}[h!]
	\centering
	\includegraphics[width = 0.85\columnwidth]{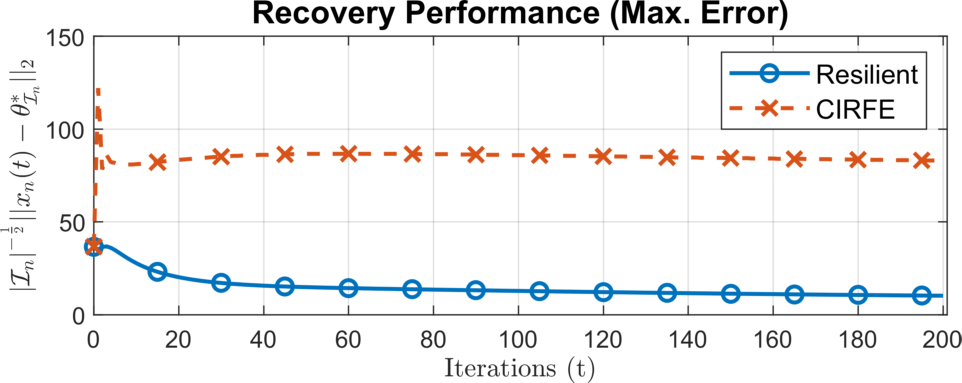}
	\caption{Evolution of the maximum (across all agents) root mean square error (RMSE) of recovery, normalized by the square root of the size of each agent's interest set.}\label{fig: performance1}
\end{figure}
 Figure~\ref{fig: performance2} shows the recovery results of our resilient algorithm and \textbf{CIRFE} after $200$ iterations, where, at each location, we report the \textit{worst} recovery result (highest RMSE) over all interested agents. 
\begin{figure}[h!]
	\centering
	\includegraphics[width = 0.85\columnwidth]{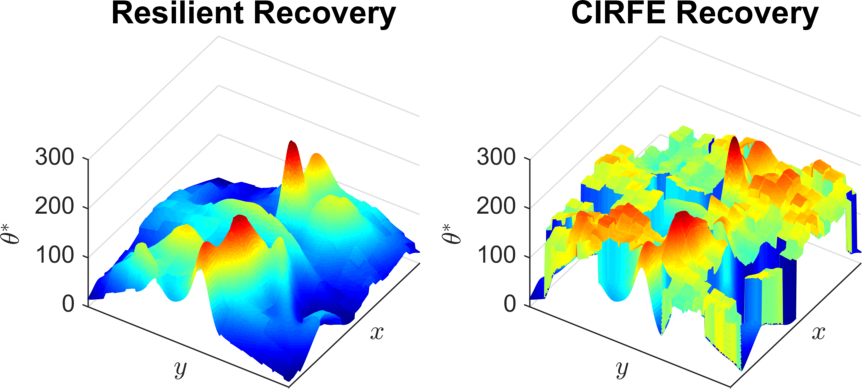}
	\caption{Recovery results from our algorithm (left) and \textbf{CIRFE}~\cite{SahuRandomFields} (right). For each location, we report the worst recovered value (highest RMSE) over all interested agents after $200$ iterations.}\label{fig: performance2}
\end{figure}
Under our algorithm, all agents resiliently recover the components of the field in which they are interested, while under~\textbf{CIRFE}, the same measurement attack prevents the agents from accurately recovering the field. 

\section{Conclusion}\label{sect: conclusion}
In this paper, we presented an algorithm for resilient distributed field recovery. A network of devices or agents measures a large, spatially distributed field. An adversary compromises a subset of the measurements, arbitrarily changing their values. Each agent processes its (possibly altered) measurements and information from its neighbors to recover certain components of the field. We presented a distributed, \textit{consensus+innovations} type algorithm for resilient field recovery. 
As long as there is enough redundancy among the uncompromised measurements in measuring the field, then, our algorithm guarantees that all of the agents' local states converge to the true values of the field components that they are interested in recovering. Finally, we illustrated the performance of our algorithm through numerical examples. 

\bibliographystyle{IEEEbib}
\bibliography{IEEEabrv,References}
\appendix
\section{Proof of Lemmas 1 and 2}
We now provide the proofs of Lemmas~\ref{lem: resilientConsensus} and~\ref{lem: avg2}. 

\subsection{Intermediate Results}
We first present intermediate results from~\cite{Kar2, ChenSAGE2, ChenSAFE}. First, the following Lemma from~\cite{Kar1} characterizes the behavior of scalar time-varying systems of the form
\begin{equation}\label{eqn: timeVaryingSystem}
	w_{t+1} = \left( 1 - r_1(t) \right) w_t + r_2(t),
\end{equation}
where $r_1(t) = \frac{c_1}{(t+1)^{\delta_1}}, r_2(t) = \frac{c_2}{(t+1)^{\delta_2}},$
$c_1, c_2 > 0$, and $0 < \delta_1 < \delta_2 < 1$. 
\begin{lemma}[Lemma 5 in~\cite{Kar1}]\label{lem: timeVaryingConvergence}
	The system in~\eqref{eqn: timeVaryingSystem} satisfies
\begin{equation}
	\lim_{t \rightarrow \infty} \left(t+1\right)^{\delta_0} w_{t+1} = 0,
\end{equation}
for every $0 \leq \delta_0 < \delta_2 - \delta_1.$
\end{lemma}

Second, the following result comes as a consequence of Lemma 3 in~\cite{ChenSAGE} and studies the convergence of scalar time-varying systems of the form
\begin{equation}\label{eqn: timeVaryingSystem2}
	w_{t+1} = \left(1 - \frac{r_1(t) c_3}{\left(\left\lvert w_{t} \right\rvert + c_5 \right)(t+1)^{\delta_3}} \right) w_t + \frac{r_1(t) c_4}{(t+1)^{\delta_4}},
\end{equation}
where $r_1(t) = \frac{c_1}{(t+1)^{\delta_1}}$, $c_3, c_4, c_5 > 0$, and $ 0 < \delta_3 < \delta_4 < \delta_1.$ 

\begin{lemma}[Lemma 3 in~\cite{ChenSAGE}]\label{lem: timeVaryingSystem1}
The system in~\eqref{eqn: timeVaryingSystem2} satisfies
\begin{equation}\label{eqn: timeVaryingConvergence}
	\lim_{t \rightarrow \infty} \left( t+1 \right)^{\delta_0} w_t = 0,
\end{equation}
for every $0 \leq \delta_0 < \delta_4 - \delta_3$. 
\end{lemma}

Finally, the following result from~Lemma 4.2 in~\cite{ChenSAFE} studies perturbations to positive definite matrices:
\begin{lemma}\label{lem: main}
 Let $A_1 \succ 0$ ($A_1 \in \mathbb{R}^{k \times k}$) be a symmetric, positive definite matrix with minimum eigenvalue $\lambda_{\min} \left(A_1 \right)$. Let $x \neq 0$, $x \in\mathbb{R}^k$, and let $y \in \mathbb{R}^k$ satisfy
$\left\lVert y \right\rVert_2 < \lambda_{\min} \left(A_1 \right) \left\lVert x \right\rVert_2.$
Then there exists $A_2 \succ 0$ such that
\begin{equation}\label{eqn: mainlem1}
	A_2 x = A_1 x + y,
\end{equation}
with a minimum eigenvalue that satisfies
\begin{equation}\label{eqn: mainlem2}
	\lambda_{\min} \left( A_2 \right) \geq \lambda_{\min} \left(A_1\right) - \frac{\left\lVert {y} \right\rVert_2}{\left \lVert {x} \right\rVert_2}.
\end{equation}
\end{lemma}

\subsection{Proof of Lemma 1}
\begin{proof}
Let
\begin{equation}\label{eqn: xHatStack}
	\widehat{\mathbf{x}}_t = \widetilde{\mathbf{x}}_t - \left(\mathbf{1}_N \otimes I_M \right) \mathbf{\overline{x}}_t,
\end{equation}
stack $\widetilde{x}_n(t) - \overline{\mathbf{x}}_t$ across all agents. For each component $m=1, \dots, M$, let \[\mathcal{J}_m = \left\{ n_{m, 1}, \dots, n_{m, \left\lvert \mathcal{J}_m \right\rvert} \right\}\] be the set of agents interested in recovering $\left[ \theta^* \right]_m$. Then, for each agent $n \in \mathcal{J}_m$, the canonical basis vector $e^\intercal_{(n-1)M + m}$ (of $\mathbb{R}^{NM}$) selects agent $n$'s estimate of $\left[\theta^* \right]_m$ from $\widetilde{\mathbf{x}}_t$: 
\[e^\intercal_{(n-1)M + m} \widetilde{\mathbf{x}}_t = \left[ \widetilde{x}_n(t) \right]_m.\]
We collect all such canonical basis row vectors for every agent $n \in \mathcal{J}_m$ in the matrix
\begin{equation}\label{eqn: qmDef}
	\mathcal{Q}_m = \left[\!\!\begin{array}{ccc}e_{\left(n_{m, 1} - 1 \right)M + m} \!\!& \cdots &\!\! e_{\left( n_{m, \left\lvert \mathcal{J}_m \right\rvert} - 1 \right)M + m} \end{array} \!\! \right]^\intercal. 
\end{equation}
Then, the $\left\lvert \mathcal{J}_m \right\vert$-dimensional vector $\widehat{\mathbf{x}}_t^m = \mathcal{Q}_m \widehat{\mathbf{x}}_t,$ collects the terms $\widehat{x}_n(t) - \overline{\mathbf{x}}_t$ from all agents $n$ interested in recovering component $m$.

From~\eqref{eqn: stackedTildeX}, we may show that $\widehat{\mathbf{x}}_t^m$ follows the dynamics
\begin{equation}\label{eqn: hatXm}
\begin{split}
	&\widehat{\mathbf{x}}_{t+1}^m = \left(I_{\left\lvert \mathcal{J}_M \right\rvert} - P_{\left\lvert \mathcal{J}_M \right\rvert, 1} - \beta_t L_m\right) \widehat{\mathbf{x}}_t^m + \\
	&\: \alpha_t \mathcal{Q}_m \left(I_{NM} - \left(\mathbf{1}_N \mathbf{1}_N^\intercal \right) \otimes \mathcal{D} \right)D^\intercal_H \mathbf{K}_t \left(\mathbf{y} - D_H \widetilde{\mathbf{x}}_t \right),
\end{split}
\end{equation}
where $P_{\left\lvert \mathcal{J}_m \right\rvert, 1} = \frac{1}{\left\lvert \mathcal{J}_m \right\rvert} \mathbf{1}_{\left\lvert \mathcal{J}_m \right\rvert} \mathbf{1}^\intercal_{\left\lvert \mathcal{J}_m \right\rvert}$ and $L_m$ is the Laplacian of $G_m$ (the subgraph of $G$ induced by the agents in $\mathcal{J}_m$). Since $G_m$ is connected, and, by definition of $\mathbf{K}_t$, we have $\left\lVert \mathbf{K}_t \left(\overline{\mathbf{y}}_t  - D_h \widetilde{\mathbf{x}}_t \right) \right\rVert_\infty \leq \gamma_t$, which means that there exists a finite constant $C_1 > 0$ such that
\begin{equation}\label{eqn: appendix1}
	C_1  \gamma_t \!\! \geq \!\! \left\lVert \mathcal{Q}_m \left(I_{NM} \!-\! \left(\mathbf{1}_N \mathbf{1}_N^\intercal \right) \otimes \mathcal{D} \right)D^\intercal_H \mathbf{K}_t \left(\mathbf{y} - D_H \widetilde{\mathbf{x}}_t \right) \right\rVert_2.
\end{equation}

Then, for $t$ large enough, we have
\begin{equation}\label{eqn: hatXm}
\begin{split}
	&\left\lVert \widehat{\mathbf{x}}_{t+1}^m \right\rVert_2 \leq \left(1 - \beta_t \lambda_2 \left(L_m \right)\right) \left\lVert \widehat{\mathbf{x}}_t^m \right\rVert_2 + C_1 \alpha_t \gamma_t.
\end{split}
\end{equation}
Since $G_m$ is connected ($\lambda_2 \left( L_m \right) > 0$), the relationship in~\eqref{eqn: hatXm} falls under the purview of Lemma~\ref{lem: timeVaryingConvergence}, which yields
\begin{equation}\label{eqn: resConsensus1}
	\lim_{t \rightarrow \infty} \left(t+1\right)^{\tau_3} \left\lVert \widehat{\mathbf{x}}_t^m \right\rVert_2 = 0,
\end{equation}
for every $0 \leq \tau_3 < \tau_\gamma + \tau_1 - \tau_2$ and every component $m = 1, \dots, M$. Let $\overline{\mathcal{Q}}$ be the matrix $\mathcal{Q}$ with all zero \textit{rows} removed, and note that $\overline{\mathcal{Q}} \widehat{\mathbf{x}}_t$ is a permutation of the vector $\left[\begin{array}{ccc} {\widehat{\mathbf{x}}_t^{1\intercal}} & \cdots & {\widehat{\mathbf{x}}_t^{M\intercal}}\end{array} \right]^\intercal.$ By the triangle inequality, we have \[\left\lVert \mathcal{Q} \widehat{\mathbf{x}}_t \right\rVert_2 = \left\lVert \overline{\mathcal{Q}} \widehat{\mathbf{x}}_t \right\rVert_2 \leq \sum_{m = 1}^M \left\lVert \widehat{\mathbf{x}}_t^m \right\rVert_2,\] which, combined with~\eqref{eqn: resConsensus1}, yields the desired result~\eqref{eqn: resilientConsensus}. 
\end{proof}

\subsection{Proof of Lemma 2}
To we prove Lemma~\ref{lem: avg2}, we require the following result. 
\begin{lemma}\label{lem: avg1}
	Let the auxiliary threshold $\overline{\gamma}_t$ be defined as
\begin{equation}\label{eqn: lineGammaDef}
	\overline{\gamma}_t = \frac{\Gamma}{(t+1)^{\tau_\gamma}} - \frac{X}{(t+1)^{\tau_3}},
\end{equation}
where $\tau_3 = \tau_\gamma + \tau_1 - \tau_2 - \epsilon_X$ for arbitrarily small $0 < \epsilon_X < \tau_1 - \tau_2,$  and, recall, $\gamma_t = \frac{\Gamma}{(t+1)^{\tau_\gamma}}$.
As long as $\lambda_{\min} \left( \mathcal{G}_{\mathcal{N}} \right) > \Delta_{\mathcal{A}}$ (resilience condition~\eqref{eqn: resilienceCondition}), then there exists $T_0 \geq 0$, and $0 < X < \infty$ such that 
\begin{enumerate}
\item $\left\lVert \mathcal{Q}\widehat{\mathbf{x}}_t \right\rVert_2 \leq \frac{X}{(t+1)^{\tau_3}}$, and
\item if, for any $T \geq T_0$, $\left\lVert \overline{\mathbf{e}}_T \right \rVert_2 \leq \overline{\gamma}_T$, then, for all $t \geq T$, $\left\lVert \overline{\mathbf{e}}_t \right\rVert_2 \leq \overline{\gamma}_t$.
\end{enumerate}
\end{lemma}
\begin{proof}[Proof of Lemma~\ref{lem: avg1}]
As a consequence of~\eqref{eqn: resilientConsensus}, from Lemma~\ref{lem: resilientConsensus}, there exists finite $X$ such that  $\left\lVert \mathcal{Q}\widehat{\mathbf{x}}_t \right\rVert_2 \leq \frac{X}{(t+1)^{\tau_3}}$. What remains is to show that the second condition holds. 

We now derive the dynamics of $\overline{\mathbf{e}}_t$. Recall that \[\overline{\mathbf{x}}_t = \mathcal{D} \left( \mathbf{1}_N^\intercal \otimes I_M \right)\widetilde{\mathbf{x}}_t\] and \[\overline{\mathbf{e}}_t = \overline{\mathbf{x}}_t - \theta^*.\] From the dynamics of $\widetilde{\mathbf{x}}_t$ (equation~\eqref{eqn: stackedTildeX}), we have that $\overline{\mathbf{x}}_t$ evolves according to
\begin{equation}\label{eqn: overlineX}
	\overline{\mathbf{x}}_{t+1} = \overline{\mathbf{x}}_t + \alpha_t \mathcal{D} \left(\mathbf{1}_N^\intercal \otimes I_M \right)D_H^\intercal \mathbf{K}_t \left( \mathbf{y} - D_H \widetilde{\mathbf{x}}_t \right).
\end{equation}
Note that, since $\tau_3 > \tau_\gamma$, for $t$ large enough, $\overline{\gamma}_t > 0$. Recall that, for each agent $n$, the matrix $Q_n$ is an $M \times M$ diagonal matrix where, the $m^{\text{th}}$ diagonal element (for $m = 1, \dots M$) is $1$ if agent $n$ is interested in recovering $\left[\theta^*\right]_m$ (i.e., $m \in \mathcal{I}_n$) and $0$ otherwise. Further recall that, for each $n$, $\widetilde{\mathcal{I}}_n \subseteq \mathcal{I}_n$. Since $\widetilde{\mathcal{I}}_n$ is the set of indices of the nonzero columns of $H_n$, by definition of $Q_n$, we have $H_n Q_n = H_n$ for each agent $n$, which means that
	$D_H \mathcal{Q} = D_H.$
Then, we may express $D_H \widetilde{\mathbf{x}}_t$ as
\begin{equation}\label{eqn: avg1proof1}
	D_H \widetilde{\mathbf{x}}_t = D_H \left( \mathbf{1}_N \otimes I_M \right) \widetilde{\mathbf{x}}_t + D_H \mathcal{Q} \widehat{\mathbf{x}}_t,
\end{equation}
where, recall, from~\eqref{eqn: xHatStack}, $\widehat{\mathbf{x}}_t = \widetilde{\mathbf{x}}_t - \left(\mathbf{1}_N \otimes I_M \right) \mathbf{\overline{x}}_t.$

Define
\begin{align}
	\mathbf{K}^{\mathcal{N}}_t &= \diag \left(\widetilde{k}_1(t), \dots, \widetilde{k}_p(t) \right), \\
	\mathbf{K}^{\mathcal{A}}_t &= \mathbf{K}_t - \mathbf{K}^{\mathcal{N}}_t. \label{eqn: ka}
\end{align}
Substituting~\eqref{eqn: avg1proof1} into~\eqref{eqn: overlineX} and performing algebraic manipulations, we may show that $\overline{\mathbf{e}}_t$ follows the dynamics
\begin{equation}\label{eqn: averageErrorDynamics}
	\begin{split}
		\overline{\mathbf{e}}_{t+1} &= \left( I_{M} - \alpha_t\mathcal{D} \sum_{p \in \mathcal{N}} k_p(t) h_p h_p^\intercal \right) \overline{\mathbf{e}}_t - \\
		& \quad \alpha_t \mathcal{D} \left( \left( \mathbf{1}_N^\intercal \otimes I_M \right) D_H^\intercal \mathbf{K}^{\mathcal{N}}_t \left(D_H\mathcal{Q}\widehat{\mathbf{x}}_t\right) + \mathbf{b}_t \right),
	\end{split}
\end{equation}
where $\mathbf{b}_t = \left(\mathbf{1}_N^\intercal \otimes I_M \right) D^\intercal_H \mathbf{K}_t^{\mathcal{A}} \left( \mathbf{y} - D_H \widetilde{\mathbf{x}}_t \right)$ captures the effect of the attack. Using the definition of $\Delta_{\mathcal{A}}$ in~\eqref{eqn: maxDisturbance} and the fact that $\left\lVert \mathbf{K}_t^{\mathcal{A}} \left( \mathbf{y} - D_H \widetilde{\mathbf{x}}_t \right) \right\rVert_\infty \leq \gamma_t$, we have that
\begin{equation}\label{eqn: aux4}
	\left\lVert \mathbf{b}_t \right\rVert_2 \leq \Delta_{\mathcal{A}} \gamma_t = \Delta_{\mathcal{A}}\left(\overline{\gamma}_t + X(t+1)^{-\tau_3} \right). 
\end{equation}
Then, we express $\mathbf{b}_t$ as
\begin{equation}\label{eqn: aux4a}
	\mathbf{b}_t = \overline{\mathbf{b}}_t + \widetilde{\mathbf{b}}_t,
\end{equation}
where $\left\lVert \widetilde{\mathbf{b}}_t \right\rVert_2 \leq  \Delta_{\mathcal{A}} \left(X(t+1)^{\tau_3}\right)$ and $\left\lVert \overline{\mathbf{b}}_t \right\rVert_2 \leq \Delta_{\mathcal{A}} \overline{\gamma}_t$.

Now, we study the evolution of $\left\lVert \overline{\mathbf{e}}_t \right\rVert_2$. We show that, for $T_0$ large enough, if, for some $T \geq T_0$, $\left\lVert \overline{\mathbf{e}}_T \right \rVert_2 \leq \overline{\gamma}_T$, ten, for all $t \geq t$, $\left\lVert \overline{\mathbf{e}}_t \right\rVert_2 \leq \overline{\gamma}_t$. Applying the triangle inequality to the noncompromised measurements $p \in \mathcal{N}$, we have
\begin{align}
	\left\lvert y^{(p)} - {h_p^c}^\intercal x_n(T) \right\rvert &\leq \left\lvert h_p^\intercal \left(Q_n \left(\overline{\mathbf{x}}_T - x_n(T) \right) - \overline{\mathbf{e}}_T \right) \right\rvert, \\
	& \leq \left\lVert \overline{\mathbf{e}}_T\right\rVert_2 + \left\lVert \mathcal{Q} \widehat{\mathbf{x}}_T \right\rVert_2, \\
	& \leq \overline{\gamma}_t + \frac{X}{(T+1)^{\tau_3}} = \gamma_t.
\end{align}
That is, $k_p(T) = 1$, for all uncompromised measurements  $p \in \mathcal{N}$.  From~\eqref{eqn: averageErrorDynamics}, we then have
\begin{equation}
	\begin{split}
		&\overline{\mathbf{e}}_{T+1} = \overline{\mathbf{e}}_{T} - \alpha_T \mathcal{D} \left(\mathcal{G}_{\mathcal{N}} \overline{\mathbf{e}}_T + \overline{\mathbf{b}}_T + \widetilde{\mathbf{b}}_T \right) - \\
		&\quad \alpha_T \mathcal{D} \left(\mathbf{1}_N \otimes I_M \right)D_H^\intercal\mathbf{K}_T^{\mathcal{N}}D_H \mathcal{Q} \widetilde{\mathbf{x}}_T,
	\end{split}
\end{equation}
where $\mathcal{G}_{\mathcal{N}} = \sum_{p \in \mathcal{N}} h_p h_p^\intercal$. Since $\Delta_{\mathcal{A}} \leq \left \lvert \mathcal{A} \right\rvert$ and $\left\lvert \mathcal{A} \right\rvert + \left\lvert \mathcal{N} \right \rvert = P$, we have
\begin{equation}\label{eqn: aux6a}
	\begin{split}
		\left\lVert \overline{\mathbf{e}}_{T+1} \right\rVert_2 & \leq \left\lVert  \left( I_{M} - \alpha_T\mathcal{D} \mathcal{G}_{\mathcal{N}} \right) \overline{\mathbf{e}}_{T} + \alpha_T \mathcal{D} \overline{\mathbf{b}}_{T} \right\rVert_2 + \\
		& \quad\frac{\alpha_T P X}{\underline{J}(T+1)^{\tau_3}},
	\end{split}
\end{equation}
where $\underline{J} = \min_{m = 1, \dots, M} \left\lvert \mathcal{J}_m \right \rvert.$

Since $\left\lVert \overline{\mathbf{e}}_T \right \rVert_2 \leq \overline{\gamma}_T$, there exists $\overline{\mathbf{e}}_T^*$ with $\left\lVert \overline{\mathbf{e}}_T \right\rVert_2 = \overline{\gamma}_T$, such that 
\begin{equation}\label{eqn: aux7}
	\begin{split}
		\left\lVert \overline{\mathbf{e}}_{T+1} \right\rVert_2 & \leq \left\lVert  \left( I_{M} - \alpha_T\mathcal{D} \mathcal{G}_{\mathcal{N}} \right) \overline{\mathbf{e}}^*_{T} + \alpha_T \mathcal{D} \overline{\mathbf{b}}_{T} \right\rVert_2 + \\
		& \quad\frac{\alpha_T P X}{\underline{J} (T+1)^{\tau_3}}.
	\end{split}
\end{equation}
By Lemma~\ref{lem: main}, there exists $\mathcal{G}^*_{T} \succ 0$ such that $\mathcal{G}^*_{T} \overline{\mathbf{e}}_{T}^* = \mathcal{G}_{\mathcal{N}} \overline{\mathbf{e}}_{T}^* + \overline{\mathbf{b}}_T$, with minimum eigenvalue
\begin{equation}\label{eqn: kappaDef}
	\lambda_\min \left( \mathcal{G}^*_{T, \omega} \right) \geq \kappa = \lambda_{\min} \left( \mathcal{G}_{\mathcal{N}} \right) - \Delta_{\mathcal{A}}.
\end{equation}
Substituting in~\eqref{eqn: aux7}, we then have, for some finite $C_2 > 0$.
\begin{equation}\label{eqn: aux8}
	\left\lVert \overline{\mathbf{e}}_{T+1} \right\rVert_2  \leq \left(1 - \alpha_T \kappa C_2 \right)\overline{\gamma}_{T, \omega} + \frac{\alpha_T P X}{\underline{J} (T+1)^{\tau_3}}.
\end{equation}

Using~\eqref{eqn: aux8}, we now show that $\left\lVert \overline{\mathbf{e}}_{T+1} \right\rVert_2  \leq \overline{\gamma}_{T+1}$. It suffices to show that $\left\lVert \overline{\mathbf{e}}_{T+1} \right\rVert_2  \leq \left(\frac{T+1}{T+2}\right)^{\tau_\gamma} \overline{\gamma}_T$. By definition of $\overline{\gamma}_T$, for any $0 < \overline{\Gamma} < \Gamma$, there exists a sufficiently large finite $T$ such that $\overline{\gamma}_t > \frac{\overline{\Gamma}}{(t+1)^{\tau_{\gamma}}}$. As a consequence, we may show that, for sufficiently large finite $T$,~\eqref{eqn: aux8} becomes
\begin{equation}\label{eqn: aux9}
	\left\lVert \overline{\mathbf{e}}_{T+1} \right\rVert_2  \leq \left(1 - \alpha_{T} \rho_{T} \right)\overline{\gamma}_{T},
\end{equation}
where \begin{equation}\label{eqn: rhoDef} \rho_{T} = \kappa C_2 - \frac{P X}{\underline{J}\overline{\Gamma} (T+1)^{\tau_3 - \tau_\gamma}}. \end{equation} The second term in $\rho_{T}$ decays to $0$ as $T$ increases, so, for sufficiently large $T$, $\rho_{T} \geq 0$. 

To proceed, we show that 
\begin{equation}\label{eqn: aux10}
	1 - \alpha_T \rho_T \leq \left(\frac{T+1}{T+2}\right)^{\tau_\gamma}
\end{equation}
 for sufficiently large $T$. Using the inequalities $(1-x) \leq e^{-x}$ for  $x \geq 0$ and $\log \left(\frac{T + 1} {T + 2} \right) \geq 1 -\frac{T + 2} {T + 1} = -\frac{1}{T + 1}$, a sufficient condition for~\eqref{eqn: aux10} is
\begin{equation}\label{eqn: aux11}
	\alpha_T \rho_{T} \geq \frac{\tau_\gamma}{T+1}.
\end{equation}
In the definition of $\rho_{T}$~\eqref{eqn: rhoDef}, the second term decays to $0$ as $T$ increases, which means that, for any constant $0 < \overline{\rho} < \kappa C_2$, there exists sufficiently large finite $T$ such that $\rho_T > \overline{\rho}$. Then, the sufficient condition~\eqref{eqn: aux11} becomes
\begin{equation}\label{eqn: aux12}
	\alpha_T \overline{\rho} \geq \frac{\tau_\gamma}{T+1},
\end{equation}
which is satisfied for all $T \geq \left(\frac{\tau_\gamma}{a \overline{\rho} }\right)^{\frac{1}{{1-\tau_1}}}- 1.$ Thus, there exists $T_0$ sufficiently large such that, if $\left\lVert \overline{\mathbf{e}}_{T} \right\rVert_2 \leq \overline{\gamma}_{T}$ for some $T \geq T_0$, we have $\left\lVert \overline{\mathbf{e}}_{T+1} \right\rVert_2 \leq \overline{\gamma}_{T+1}$. The same analysis holds for all $t = T+2, T+3, \dots$, which completes the proof. 
\end{proof}

We now prove Lemma~\ref{lem: avg2}. 

\begin{proof}[Proof of Lemma~\ref{lem: avg2}]
	As a consequence of Lemma~\ref{lem: avg1}, there exists $T_0 \geq 0$ such that, if at any $T \geq T_0$,  $\left\lVert \overline{\mathbf{e}}_T \right \rVert_2 \leq \overline{\gamma}_T$, then, for all $t \geq T$,  $\left\lVert \overline{\mathbf{e}}_t \right \rVert_2 \leq \overline{\gamma}_t$. If such a $T$ exists, then we have  $\left\lVert \overline{\mathbf{e}}_t \right \rVert_2 \leq \overline{\gamma}_t < \gamma_t,$ which means that $ \overline{\mathbf{e}}_t $ satisfies~\eqref{eqn: averageConv}.

If no such $T$ exists, then, for all $t \geq T_0$,  $\left\lVert \overline{\mathbf{e}}_t \right \rVert_2 > \overline{\gamma}_t.$ Define
\begin{equation}\label{eqn: avgConv2}
	\widehat{K}_{t} = \frac{\overline{\gamma}_{t} + \frac{X}{(t+1)^{\tau_3}}}{\left\lVert \overline{\mathbf{e}}_{t} \right \rVert_2  + \frac{X}{(t+1)^{\tau_3}} }. 
\end{equation}
Using the fact that $\left\lVert \mathcal{Q}\widehat{\mathbf{x}}_t \right\rVert_2 \leq \frac{X}{(t+1)^{\tau_3}}$ and applying the triangle inequality, we may show that $\widehat{K}_t < k_p (t)$ for all $p \in \mathcal{N}$. Rearranging~\eqref{eqn: avgConv2}, we have
\begin{equation}\label{eqn: avgConv3}
	\gamma_t = \widehat{K}_t \left( \left \lVert \overline{\mathbf{e}}_t \right\rVert_2 + \frac{X}{(t+1)^{\tau_3}} \right).
\end{equation}
Recall that, in the dynamics of $\overline{\mathbf{e}}_t$~\eqref{eqn: averageErrorDynamics}, the vector \[\mathbf{b}_t = \left(\mathbf{1}_N^\intercal \otimes I_M \right) D^\intercal_H \mathbf{K}_t^{\mathcal{A}} \left( \mathbf{y} - D_H \widetilde{\mathbf{x}}_t \right)\] represents the effect of the attack and satisfies $\left\lVert \mathbf{b}_t \right\rVert_2 \leq \Delta_{\mathcal{A}} \gamma_t$. Using~\eqref{eqn: avgConv3}, we then partition $\mathbf{b}_t$ (differently from the partition described by~\eqref{eqn: aux4a}) as $\mathbf{b}_t = \overline{\mathbf{b}}_t + \widetilde{\mathbf{b}}_t,$ where $\left\lVert \overline{\mathbf{b}}_{t} \right\rVert_2 \leq \widehat{K}_{t} \Delta_{\mathcal{A}} \left\lVert \overline{\mathbf{e}}_{t} \right\rVert_2$ and $\left\lVert \widetilde{\mathbf{b}}_{t} \right\rVert_2 \leq \frac{ \widehat{K}_{t} \Delta_{\mathcal{A}}X}{(t+1)^{\tau_3}}.$

Substituting for this partition of $\overline{\mathbf{b}}_t$ and using the fact that $\widehat{K}_t > k_p(t)$ for all $p \in \mathcal{N}$, we may show, from~\eqref{eqn: averageErrorDynamics}, that
\begin{equation}\label{eqn: avgConv8}
	\begin{split}
		\left\lVert \overline{\mathbf{e}}_{t+1} \right\rVert_2 & \leq \left\lVert  \left( I_{M} - \alpha_t \widehat{K}_t \mathcal{D} {\mathcal{G}}_{\mathcal{N}} \right) \overline{\mathbf{e}}_{t, \omega} + \alpha_t \mathcal{D} \overline{\mathbf{b}}_{t} \right\rVert_2 + \\
		& \quad \frac{\alpha_t P X}{\underline{J} (t+1)^{\tau_3}}.\\
	\end{split}
\end{equation}
As a consequence of Lemma~\ref{lem: main}, there exists $\mathcal{G}^*_t \succ 0$ such that $\mathcal{G}^*_t \overline{\mathbf{e}}_t = \widehat{K}_t\mathcal{G}_{\mathcal{N}} \overline{\mathbf{e}}_t + \overline{\mathbf{b}}_t$ with a minimum eigenvalue that satisfies $\lambda_{\min} \left(\mathcal{G}^*_t \right) \geq \widehat{K}_t \kappa$, where $\kappa = \lambda_{\min}\left(\mathcal{G}_{\mathcal{N}} \right) - \Delta_{\mathcal{A}}$ (see~\eqref{eqn: kappaDef}). Substituting for $\mathcal{G}^*_t$ into~\eqref{eqn: avgConv8} and performing algebraic manipulations, we have, for finite $C_2 > 0$ 
\begin{equation}\label{eqn: avgConv9}
	\begin{split}
		\left\lVert \overline{\mathbf{e}}_{t+1} \right\rVert_2 \leq &\left(1 - \alpha_t \widehat{K}_t C_3 \kappa \right)\left\lVert \overline{\mathbf{e}}_{t} \right\rVert_2  + \frac{\alpha_t P X}{\underline{J}(t+1)^{\tau_3}}.\\
	\end{split}
\end{equation}
Using the fact that $\widehat{K}_t >{\Gamma}\left({ (t+1)^{\tau_\gamma} \left(\left\lVert \overline{\mathbf{e}}_{t} \right\rVert_2 + X \right) }\right)^{-1}$, the relation in~\eqref{eqn: avgConv8} falls under the purview of Lemma~\ref{lem: timeVaryingSystem1}, and we have
	$\lim_{t \rightarrow \infty} \left( t+1 \right)^{\tau_0} \left\lVert \overline{\mathbf{e}}_{t} \right\rVert_2 = 0,$
for every $0 \leq \tau_0 < \tau_3 - \tau_{\gamma} = \tau_1 - \tau_2 - \epsilon_X$. Taking $\epsilon_X$ arbitrarily close to $0$ yields the desired result~\eqref{eqn: averageConv} and completes the proof.
\end{proof}
\end{document}